\documentclass[a4paper,twopage,reqno,12pt]{amsart}
\usepackage[top=30mm,right=30mm,bottom=30mm,left=30mm]{geometry}

\usepackage{graphicx}
\usepackage{amsmath}
\usepackage{amssymb}
\usepackage{amsfonts}
\usepackage{amsthm}
\usepackage{amstext}
\usepackage{amsbsy}
\usepackage{amsopn}
\usepackage{amscd}
\usepackage{enumerate}
\usepackage{color}
\usepackage[colorlinks]{hyperref}
\usepackage[hyperpageref]{backref}
\usepackage{multirow}
\usepackage{float}
\usepackage[norelsize]{algorithm2e}
\usepackage{url}

\newtheorem{theorem}{Theorem}[section]

\newtheorem{lemma}[theorem]{Lemma}
\newtheorem{proposition}[theorem]{Proposition}

\theoremstyle{definition}

\numberwithin{equation}{section}

\newcommand{\Soc}{\mathrm{Soc}}

\newcommand{\Aut}{\mathrm{Aut}}

\newcommand{\Out}{\mathrm{Out}}

\newcommand{\Dmc}{\mathcal{D}}
\newcommand{\Bmc}{\mathcal{B}}
\newcommand{\Pmc}{\mathcal{P}}

\renewcommand{\leq}{\leqslant}
\renewcommand{\geq}{\geqslant}

\renewcommand{\mod}[1]{\ (\mathrm{mod}{\ #1})}
\newcommand{\imod}[1]{\allowbreak\mkern4mu({\operator@font mod}\,\,#1)}

\begin{document}
 \title{On flag-transitive automorphism groups of symmetric designs}

 \author[S.H. Alavi]{Seyed Hassan Alavi}%
 \thanks{Corresponding author: S.H. Alavi}
 \address{Seyed Hassan Alavi, Department of Mathematics, Faculty of Science, Bu-Ali Sina University, Hamedan, Iran.
 }%
 \email{alavi.s.hassan@basu.ac.ir and  alavi.s.hassan@gmail.com (G-mail is preferred)}
 \author[N. Okhovat]{Narges Okhovat}%
 \address{Narges Okhovat, Department of Mathematics, Faculty of Science, Bu-Ali Sina University, Hamedan, Iran.}%
 \email{okhovat.nargeshh@gmail.com}
 \author[A. Daneshkhah]{Asharf Daneshkhah}%
 \address{Asharf Daneshkhah, Department of Mathematics, Faculty of Science, Bu-Ali Sina University, Hamedan, Iran.
 }%
  \email{adanesh@basu.ac.ir}

 \date{\today}%

\begin{abstract}
    In this article, we study flag-transitive automorphism groups of non-trivial symmetric $(v, k, \lambda)$ designs, where $\lambda$ divides $k$ and $k\geq \lambda^2$. We show that such an automorphism group is either point-primitive of affine or almost simple type, or point-imprimitive with parameters $v=\lambda^{2}(\lambda+2)$ and $k=\lambda(\lambda+1)$, for some positive integer $\lambda$. We also provide some examples in both possibilities.
\end{abstract}

\maketitle

\section{Introduction}\label{sec:intro}

A $t$-\emph{design} $\Dmc=(\Pmc,\Bmc)$ with parameters $(v,k,\lambda)$ is an incidence structure consisting of a set $\Pmc$ of $v$ \emph{points}, and a set $\Bmc$ of $k$-element subsets of $\Pmc$, called \emph{blocks}, such that every $t$-element subset of points lies in exactly $\lambda$ blocks. The design $\Dmc$ is \emph{non-trivial} if $t < k < v-t$, and is \emph{symmetric} if $|\Bmc| = v$. By \cite[Theorem 1.1]{a:Camina94}, if $\Dmc$ is symmetric and non-trivial, then $t\leq 2$, see also \cite[Theorem 1.27]{b:Hugh-design}. Thus we study non-trivial symmetric $2$-designs with parameters $(v,k,\lambda)$ which we simply call non-trivial \emph{symmetric $(v,k,\lambda)$ designs}.
A \emph{flag} of $\Dmc$ is an incident pair $(\alpha,B)$, where $\alpha$ and $B$ are a point and a block of $\Dmc$, respectively. An \emph{automorphism} of a symmetric design $\Dmc$ is a permutation of the points permuting the blocks and preserving the incidence relation. An automorphism group $G$ of $\Dmc$ is called \emph{flag-transitive} if it is transitive on the set of flags of $\Dmc$. If $G$ leaves invariant a non-trivial partition of $\Pmc$, then $G$ is said to be \emph{point-imprimitive}, otherwise, $G$ is called \emph{point-primitive}. We here adopt the standard notation as in \cite{b:Atlas,b:Wilson} for finite simple groups of Lie type, for example, we use $PSL_{n}(q)$, $PSp_{n}(q)$, $PSU_{n}(q)$, $P\Omega_{2n+1}(q)$ and $P\Omega_{2n}^{\pm}(q)$ to denote the finite classical simple groups.
Symmetric and alternating groups on $n$ letters are denoted by $S_{n}$ and $A_{n}$, respectively. Further notation and definitions in both design theory and group theory are standard and can be found, for example in \cite{b:Dixon,b:Hugh-design,b:lander}. We also use the software \textsf{GAP} \cite{GAP4} for computational arguments.

Flag-transitive incidence structures have been of most interest. In 1961, Higman and McLaughlin \cite{a:HigMcL61} proved that a flag-transitive automorphism group of a linear space must act primitively on its points set, and then Buekenhout, Delandtsheer and Doyen \cite{a:buekenhout-1988} studied this action in details and proved that a linear space admitting a flag-transitive automorphism group (which is in fact point-primitive) is either of affine, or almost simple type.
Thereafter, a deep result \cite{a:buekenhout-1990}, namely the classification of flag-transitive finite linear spaces relying on the Classification of Finite Simple Groups (CFSG) was announced.
Although, flag-transitive symmetric designs are not necessarily point-primitive, Regueiro \cite{a:reg-reduction} proved that a flag-transitive and point-primitive automorphism group of such designs for $\lambda\leq 4$ is of affine or almost simple type, and so using CFSG, she determined all flag-transitive and point-primitive biplanes ($\lambda=2$). In conclusion, she gave a classification of flag-transitive biplanes except for the $1$-dimensional affine case \cite{t:Regueiro}. Tian and Zhou \cite{a:Zhou-lam100} proved that a flag-transitive and point-primitive automorphism group of a symmetric design with $\lambda\leq 100$ must be of affine or almost simple type. Generally, Zieschang \cite{a:zieschang-1988} proved in 1988 that if a flag-transitive automorphism group of a $2$-design with $\gcd(r,\lambda)=1$ is a point-primitive group of affine or almost simple type, and this result has been generalised by Zhuo and Zhan \cite{a:Zhou-lamcond18} for $\lambda \geq  \gcd(r, \lambda)^2$. In this paper, we study flag-transitive automorphism groups of symmetric $(v, k, \lambda)$ designs, where $\lambda$ divides $k$ and $k\geq \lambda^2$, and we show that such an automorphism group is not necessarily point-primitive:

\begin{theorem}\label{thm:main}
Let $\Dmc=(\Pmc,\Bmc)$ be a non-trivial symmetric $(v, k, \lambda)$ design with $\lambda\geq 1$, and let $G$ be a flag-transitive  automorphism group of $\Dmc$. If $\lambda$ divides $k$ and $k\geq \lambda^2$, then one of the following holds:
\begin{enumerate}[{\rm (a)}]
    \item $G$ is point-primitive of  affine or almost simple type;
    \item $G$ is point-imprimitive and $v=\lambda^{2}(\lambda+2)$ and $k=\lambda(\lambda+1)$, for some positive integer $\lambda$. In particular, if $G$ has $d$ classes of imprimitivity of size $c$, then there is a constant $l$ such that, for each block $B$  and each class $\Delta$, the size $|B \cap \Delta|$ is either $0$, or $l$, and $(c,d,l)=(\lambda^2,\lambda+2,\lambda)$ or $(\lambda+2,\lambda^{2},2)$.
\end{enumerate}
\end{theorem}

We highlight here that if $\lambda$ divides $k$, then $\gcd(k,\lambda)^2=\lambda^2>\lambda$ which does not satisfy the conditions which have been studied in \cite{a:Zhou-lamcond18,a:zieschang-1988}. Moreover, in Section \ref{sec:example}, we provide some examples to show that both possibilities in Theorem \ref{thm:main} can actually occur. 

In order to prove Theorem \ref{thm:main}(a), we apply O'Nan-Scott Theorem \cite{a:LPS-Onan-Scott} and discuss possible types of primitive groups in Section \ref{sec:prim}. We further note that our proof for part (a) relies on CFSG. To prove part (b), we use an important result by Praeger and Zhou \cite[Theorem 1.1]{a:Praeger-imprimitive}  on characterisation of imprimitive flag-transitive symmetric designs.

\subsection{Examples and comments on Theorem~\ref{thm:main}}\label{sec:example}

Here, we give some examples of symmetric $(v,k,\lambda)$ designs admitting flag-transitive automorphism groups, where $\lambda$ divides $k$ and $k\geq \lambda^2$. In Table \ref{tbl:examples}, we list some small examples of such designs with $\lambda\leq 3$. To our knowledge the design in line 2 is the only point-primitive example of symmetric designs with $v\leq 2500$ satisfying the conditions of Theorem \ref{thm:main} and this motivates the authors to investigate symmetric designs admitting symplectic automorphism groups \cite{a:ADO-PSp4}. More examples of symmetric designs admitting flag-transitive and point-imprimitive automorphism groups can be found in \cite{a:Praeger-imprimitive} and references therein.\smallskip

\noindent \textbf{Line 1.} Hussain \cite{a:hussain1945} showed that there are exactly three symmetric $(16, 6, 2)$ designs, and Regueiro proved that exactly two of such designs are flag-transitive and point-imprimitive \cite[p. 139]{a:reg-reduction}.\smallskip

\noindent \textbf{Line 2.} The symmetric design in this line  arises from the study of primitive permutation groups with small degrees. This design belongs to a class of symmetric designs with parameters $(3^m(3^m+1)/2,3^{m-1}(3^m-1)/2,3^{m-1}(3^{m-1}-1)/2)$, for some positive integer $m>1$, see \cite{a:Braic-2500-nopower,a:Dempwolff2001}.
If $m=2$, then we obtain the symmetric $(45,12,3)$ design admitting $PSp_4(3)$ or $PSp_4(3): 2$ as flag-transitive automorphism group of rank $3$, see \cite{a:Braic-2500-nopower}.\smallskip

\noindent \textbf{Lines 3-4.} Mathon and Spence \cite{a:mathon-1996} constructed $2616$ pairwise non-isomorphic symmetric $(45,12,3)$ designs with  non-trivial automorphism groups. Praeger \cite{a:praeger-45point} proved that there are exactly two flag-transitive  symmetric $(45, 12, 3)$ designs, exactly one of which admits a point-imprimitive group, and this example satisfies Line 4, but not Line 3.\smallskip

\begin{table}
\centering
\small
\caption{Some symmetric designs satisfying the conditions in Theorem \ref{thm:main}} \label{tbl:examples}
\begin{tabular}{lllllllllllp{3cm}}
    \hline
    Line & $v$ & $k$ &$\lambda$ & $c$ & $d$ & $l$ & Case & Examples & Reference & Comments \\
    \hline
    $1$ & $16$ & $6$ & $2$ & $4$ & $4$ & $2$ & (b) & $2$ & \cite{a:hussain1945},  \cite{a:reg-reduction}& imprimitive\\
    $2$ & $45$ & $12$ & $3$ & - & - & - & (a) & $1$ & \cite{a:Braic-2500-nopower}& primitive\\
    $3$ & $45$ & $12$ & $3$ & $5$ & $9$ & $2$ & (b) & None & \cite{a:praeger-45point} & imprimitive\\
    $4$ & $45$ & $12$ & $3$ & $9$ & $5$ & $3$ & (b) & 1 & \cite{a:praeger-45point} & imprimitive\\
    \hline
\end{tabular}
\end{table}

\section{Preliminaries}\label{sec:pre}

In this section, we state some useful facts in both design theory and group theory.

\begin{lemma}\label{lem:six}{\rm \cite[Lemma 2.1]{a:ABD-PSL2}}
Let $\Dmc$ be a symmetric $(v,k,\lambda)$ design, and let $G$ be a flag-transitive automorphism group of $\Dmc$. If $\alpha$ is a point in $\Pmc$ and $H:=G_{\alpha}$, then
\begin{enumerate}[\rm (a)]
    \item $k(k-1)=\lambda(v-1)$;
    \item $k\mid |H|$ and $\lambda v<k^2$.
\end{enumerate}
\end{lemma}

\begin{lemma}{\rm \cite[Corollary 4.3]{a:AB-Large-15}}\label{lem:bound}
Let $T$ be a  finite simple classical group of dimension $n$ over a finite field $\mathbb{F}_q$ of size $q$. Then
\begin{enumerate}[{\rm (a)}]
    \item If $T=PSL_n(q) $ with $n\geq 2$, then $ |T|>q^{n^2-2} $;
    \item If $T=PSU_n(q) $ with $n\geq 3$, then $ |T|>(1-q^{-1})q^{n^2-2} $;
    \item If $T=PSp_n(q) $ with $n\geq 4$, then $|T|>q^{\frac{1}{2}n(n+1)}/(2\alpha)$, where $\alpha=\gcd(2, q-1)$;
    \item If $T=P\Omega ^{\epsilon}_n(q) $ with $n\geq 7$, then $ |T|> q^{\frac{1}{2}n(n-1)}/(4\beta)$, where $\beta=\gcd(2, n)$.
\end{enumerate}
\end{lemma}

\begin{lemma}\label{lem:out}
Let $T$ be a non-abelian finite simple group satisfying
\begin{align}\label{eq:out}
|T|<8\cdot |\Out(T)|^{3}.
\end{align}
Then $T$ is isomorphic to $A_{5}$ or $A_{6}$.
\end{lemma}
\begin{proof}
If $T$ is a sporadic simple group or an alternating group $A_{n}$ with $n\geq 7$, then $|\Out(T)|\in\{ 1,2\}$, and so by~\eqref{eq:out}, we must have $|T|<64$, which is a contradiction. Note that the alternating groups $A_{5}$ and $A_{6}$ satisfy \eqref{eq:out} as claimed. Therefore, we only need to consider the case where $T$ is a finite simple group of Lie type. In what follows, we discuss each case separately.

Let $T=PSL_{n}(q)$ with $q=p^{a}$ and $n\geq 2$. If $n=2$, then $q\geq 4$ and $|\Out(T)|=a\cdot \gcd(2,q-1)$, and so by Lemmas~\ref{lem:bound}(a) and \eqref{eq:out}, we have that  $q^{2}<|PSL_{2}(q)|<8a^{3}\cdot \gcd(2,q-1)^{3}\leq 64a^{3}$. Thus, $q^{2}<64a^{3}$. This inequality holds only for $(p,a)\in\{(2,1),(2,2), (2,3),(2,4), (2,5),(2,6),(2,7), (3,1), (3,2),(3,3),(5,1),(7,1)\}$. Note in this case that $q\geq 4$, and hence by \eqref{eq:out}, we conclude that $T$ is either $PSL_{2}(4)\cong PSL_{2}(5)\cong A_{5}$, or $PSL_{2}(9)\cong A_{6}$, as claimed. If $n=3$, then by Lemma~\ref{lem:bound}(a), we have that $q^{7}<64a^{3}\cdot \gcd(3,q-1)^{3}< 64a^{3}q^{3}$, and so $q^{4}<64a^{3}$. If $q$ would be odd, then we would have $3^{4a}<64a^{3}$, which is impossible. If $q=2^{a}$, then $2^{a}<64a^{3}$ would hold only for $a=1,2$. Therefore, $T$ is isomorphic to $PSL_{3}(2)$ or $PSL_{3}(4)$. These simple groups do not satisfy \eqref{eq:out}. If $n\geq 4$, then \eqref{eq:out} implies that $ q^{11}<64a^{3}$, but this inequality has no possible solution.

Let $T=PSU_{n}(q)$ with $q=p^{a}$ and $n\geq 3$. By Lemma~\ref{lem:bound}(b), we have that $|T|> (1-q^{-1})q^{n^2-2}$, and so \eqref{eq:out} follows that   $(1-q^{-1})q^{n^2-2}<64a^{3}\cdot \gcd(n, q+1)^{3}$. If $n=3$, then $(1-q^{-1})q^7<64a^{3}\cdot \gcd(n,q+1)^{3}$, and so $ q^6<27\cdot 64 a^3 $. This inequality holds only for $(p,a)\in\{(2,1), (2,2),(3,1) \}$. Note that $PSU_{3}(2)$ is not simple. Therefore, $T$ is isomorphic to $PSU_{3}(3)$ or $PSU_{3}(4)$. These simple groups do not satisfy \eqref{eq:out}. If $n\geq 4$, then since $(q+1)^3<4\cdot q^3(q-1)$, we would have $q^{n^2-3}<64a^{3}\cdot \gcd(n, q+1)^{3}/(q-1)<4\cdot 64a^{3}(q+1)^{3}/4(q-1)<4\cdot 64 a^{3}q^{3}$, and so $q^{n^2-6}<4\cdot 64 a^{3}$, and hence $q^{10}<4\cdot 64a^{3}$, which is impossible.

Let $T=PSp_{n}(q)$ with $q=p^{a}$ and $n\geq 4$. By Lemma~\ref{lem:bound}(c), we observe that $|T|> q^{\frac{1}{2}n(n+1)}/2\gcd(2,q-1)\geq q^{\frac{1}{2}n(n+1)}/4$. By \eqref{eq:out}, we have that $q^{10}\leq q^{\frac{1}{2}n(n+1)}<4\cdot 64a^{3}$, and so $q^{10}<4\cdot 64a^{3}$,  which is impossible.

Let $T=P\Omega_n(q)$ with $q=p^{a}$ odd and $n\geq 7$. Then we conclude by  Lemma~\ref{lem:bound}(d) that $|T|>q^{\frac{1}{2}n(n-1)}/8$. Since $|\Out(T)|=2a$ and $n\geq 7$, it follows from \eqref{eq:out} that $q^{21}<8^3a^{3}$, which is impossible.

Let $T=P\Omega^{\epsilon}_n(q)$ with $q=p^{a}$ and $n\geq 8$ and $\epsilon=\pm$. It follows from Lemma~\ref{lem:bound}(d) that $|T|>q^{\frac{1}{2}n(n-1)}/8$. Note that $|\Out(T)|\leq 6a\cdot \gcd(4,q^{\frac{n}{2}}-\epsilon)\leq 24 a$. Then \eqref{eq:out} implies that $q^{28}<8^2\cdot 24^3a^{3}$, which is impossible.

Let $T$ be one of the finite exceptional groups $F_4(q)$, $E_6(q)$, $E_7(q)$, $E_8(q)$, $^2\!F_4(q)$ ($q=2^{2m+1}$), $^3\!D_4(q)$ and $^2\!E_6(q)$. Then  $|T|>q^{20} $, and so \eqref{eq:out} implies that $q^{20}<8\cdot 2^3\cdot 3^3a^{3}$, which is impossible. If $T=G_2(q) $ with $q=p^{a} \neq 2$. Then by \eqref{eq:out}, we have that $q^{12}< q^{6}(q^2-1)(q^6-1)<8\cdot 2^3a^{3}$, and so $q^{12}<8\cdot 2^3a^{3} $, which is impossible. Similarly, if $T$ is one of the groups $^{2}\!B_2(q) $ with $q=2^{2m+1}$ and $^{2}\!G_2(q) $ with $q=3^{2m+1}$, then $|T|>q^{4}$, and so \eqref{eq:out} implies that $q^{4}<8a^3$, which is impossible.
\end{proof}


\section{Point-primitive designs}\label{sec:prim}

In what follows, we assume that  $\Dmc=(\Pmc,\Bmc)$ is a non-trivial symmetric $(v, k, \lambda)$ design admitting a flag-transitive and  point-primitive automorphism group $G$. Let also $\lambda$ divide $k$ and $k\geq \lambda^2$ and set $t:=k/\lambda$. Notice that $\lambda <k$, and so $t \geq 2$. We moreover observe by Lemma~\ref{lem:six}(a) that
\begin{align}
k&=\dfrac{v+t-1}{t};\label{eq:k} \\
\lambda &=\dfrac{v+t-1}{t^2}. \label{eq:lam}
\end{align}
Since also $G$ is a primitive permutation group on $\Pmc$, then by O'Nan-Scott Theorem \cite{a:LPS-Onan-Scott}, $G$ is of one of the following types:
\begin{enumerate}[{\rm (a)}]
\item Affine;
\item Almost simple;
\item Simple diagonal;
\item Product;
\item Twisted wreath product.
\end{enumerate}

\subsection{Product and twisted wreath product type}

In this section, we assume that $G$ is a primitive group of product type on $\Pmc$, that is to say, $G\leq H \wr S_{\ell}$, where $H$ is of almost simple or diagonal type on the set $\Gamma$ of size $m:=|\Gamma |\geq 5$ and $\ell\geq 2$. In this case, $\Pmc = \Gamma^{\ell}$.

\begin{lemma}\label{lem:sub-prod}
Let $G$ be a flag-transitive point-primitive automorphism group of product type. Then $k$ divides $\lambda \ell(m-1)$.
\end{lemma}
\begin{proof}
See the proof of \rm{Lemma~4} in \cite{a:reg-reduction}.
\end{proof}

\begin{proposition}\label{prop:prod}
If $\Dmc=(\Pmc,\Bmc)$ is a non-trivial symmetric $(v, k, \lambda)$ design admitting a flag-transitive and point-primitive automorphism group $G$, where $\lambda$ divides $k$ and $k\geq \lambda^2$, then $G$ is not of product type.
\end{proposition}
\begin{proof}
Assume the contrary. Suppose that $G$ is of product type. Then $v=m^{\ell}$. Note by Lemma~\ref{lem:sub-prod} that $k$ divides $\lambda\ell(m-1)$, and so $t=k/\lambda$ divides $\ell(m-1)$. We also note by Lemma \ref{lem:six}(b) that $\lambda v <k^{2}$. Then $v<\lambda t^2$, and since $\lambda\leq t$, we have that $v < t^3$. Recall that $t$ divides $\ell(m-1)$. Hence
\begin{align}\label{eq:prod}
m^\ell<\ell^3(m-1)^3.
\end{align}
Then $m^{\ell}<\ell^{3} m^{3}$, or equivalently, $m^{\ell-3}<\ell^{3}$. Since $m\geq 5$, it follows that $5^{\ell-3}<\ell^{3}$, and this is true for $2\leq \ell\leq 6$. If $\ell=6$, then since $m^{6-3}<6^3$, we conclude that $m=5$, but $(m,\ell)=(5,6)$ does not satisfy \eqref{eq:prod}. Therefore, $2\leq \ell\leq 5$.

Suppose first that $\ell=5$. Then by~\eqref{eq:prod}, we have that $m^{5}<5^{3}(m-1)^{3}$, and so $5\leq m\leq 9$. It follows from \eqref{eq:k} that $t$ divides $m^{5}-1$. For each $5\leq m\leq 9$, we can obtain divisors $t$ of $m^{5}-1$. Note by \eqref{eq:lam} that $t^{2}$ must divide $m^{5}-t+1$. This is true only for $m=7$ when $t=2$ or $6$ for which $(v,k,\lambda)=(16807, 8404, 4202)$ or $(16807, 2802, 467)$, respectively. Since $\lambda^{2}\leq k$, these parameters can be ruled out.

Suppose that $\ell=4$. Then by~\eqref{eq:prod}, we have that $m^{5}<4^{3}(m-1)^{3}$, and so $5\leq m\leq 9$. By the same argument as in the case where $\ell=5$, by \eqref{eq:k} and \eqref{eq:lam}, we obtain possible parameters $(m,t,v,k,\lambda)$ as in Table~\ref{tbl:l4}. Note by Lemma~\ref{lem:sub-prod} that $k$ must divide $4\lambda(m-1)$, and this is not true, for all parameters in Table~\ref{tbl:l4}.
\begin{table}[h]
    \centering
    \small
    \caption{Possible values for $(m,t,v,k,\lambda)$ when $\ell=4$.}\label{tbl:l4}
    \begin{tabular}{lllll}
        \hline
        $m$ & $t$ & $v$ & $k$ &  $ \lambda $ \\
        \hline
        13 & 51 & 28561 & 561 & 11\\
        31 & 555 & 923521 & 1665 & 3\\
        47 & 345 & 4879681 & 14145 & 41\\
        57 & 416 & 10556001 & 25376 & 61\\
        \hline
    \end{tabular}
\end{table}

Suppose now that $\ell=3$. We again apply Lemma~\ref{lem:sub-prod} and conclude that $t$ divides $3(m-1)$. Then there exists a positive integer $x$ such that $3(m-1)=tx$, and so $m=(tx+3)/3$. By \eqref{eq:lam}, we have that
\begin{align*}
\lambda=\frac{m^2+t-1}{t^2}=\frac{t^2x^3+9tx^2+27x+27}{27t}.
\end{align*}
Then $27\lambda t=t^2x^3+9tx^2+27x+27$. Therefore, $t$ must divide $27x+27$, and so $ty=27x+27$, for some positive integer $y$. Thus,
\begin{align}\label{eq:lam3}
\lambda=\frac{t(ty-27)^3+9\cdot 27(ty-27)^2+27^{3}y}{27^{4}},
\end{align}
for some positive integers $t$ and $y$. Since $\lambda^2\leq k $, we have that $\lambda\leq t$, and so
\begin{align}\label{eq:l3}
t(ty-27)^3+9\cdot 27(ty-27)^2+27^{3}y\leq 27^{4}t.
\end{align}
If $y\geq 32$, then $t(ty-27)^3+9\cdot 27(ty-27)^2+27^{3}y\geq t(32t-27)^3+9\cdot 27(32t-27)^2+32\cdot 27^{3}>27^{4}t$, for $t\geq 2$. Thus $1\leq y\leq 31$, and so by \eqref{eq:l3}, we conclude that $2\leq t\leq 107$. For each such $y$ and $t$, by straightforward calculation, we observe that $\lambda$ as in \eqref{eq:lam3} is not a positive integer.

Suppose finally that $\ell=2$. Recall by Lemma~\ref{lem:sub-prod} that $t$ divides $2(m-1)$. Then $2(m-1)=tx$ for some positive integer $x$, and so $m=(tx+2)/2$. It follows from \eqref{eq:lam} that $\lambda =(tx^2+4x+4)/4t$, or equivalently, $4t\lambda=tx^2+4x+4$. This shows that $t$ divides $4x+4$, and so $ty=4x+4$, for some positive integer $y$. Therefore, $4^{3}\lambda=(ty-4)^2+16y$. Since $\lambda^2\leq k $, we have that $\lambda\leq t$, and so $(ty-4)^2+16y\leq4^3 t$. If $y\geq 6$, then $(6t-4)^2+6\cdot 16\leq4^3 t$, which has no possible solution for $t$. Thus $1\leq y\leq 5$. Since also $(t-4)^2+16\leq4^3 t$, we conclude that $2\leq t\leq 71$, and so \eqref{eq:k} and \eqref{eq:lam} imply that
\begin{align*}
k= \frac{t(t^2y^2-8ty+16y+16)}{64} \text{ and } \lambda=\frac{(ty-4)^2+16y}{64},
\end{align*}
where $2\leq t\leq 71$ and $1\leq y\leq 5$. For these values of $t$ and $y$, considering the fact that $m\geq 5$, $k\geq \lambda^{2}$ and $\lambda$ divides $k$, we obtain $(v,k,\lambda)= (121, 25, 5)$ or $(441, 56, 7)$ respectively when $(t,y)=(5,4)$ or $(8,3)$. These possibilities can be ruled out by \cite{a:Braic-2500-nopower} or \cite[Theorem~1.1]{a:Zhou-lam100}.
\end{proof}

\begin{proposition}\label{prop:tw}
If $\Dmc=(\Pmc,\Bmc)$ is a non-trivial symmetric $(v, k, \lambda)$ design admitting a flag-transitive and point-primitive automorphism group $G$, where $\lambda$ divides $k$ and $k\geq \lambda^2$, then $G$ is not of twisted wreath product type.
\end{proposition}
\begin{proof}
If $G$ would be of twisted wreath product type, then by \cite[Remark 2(ii)]{a:LPS-Onan-Scott}, it would be contained in the wreath product $H \wr S_m$ with $H = T \times T$ of simple diagonal type, and so $G$ would act on $\Pmc$ by product action, and this contradicts Proposition~\ref{prop:prod}.
\end{proof}

\subsection{Simple diagonal type}\label{sec:sd}

In this section, we suppose that $G$ is a primitive group of diagonal type. Let $M=\Soc(G)=T_1\times\ldots\times T_m $, where $T_i\cong T$ is a non-abelian finite simple group, for $i=1,\ldots, m$. Then $G$ may be viewed as a subgroup of $M\cdot(\Out(T)\times S_m)$. Here, $G_\alpha$ is isomorphic to a subgroup of $\Aut(T ) \times S_m$ and $M_{\alpha} \cong T$ is a diagonal subgroup of $M$, and so $|\Pmc|=|T|^{m-1}$.

\begin{lemma}\label{lem:sub-sd}
Let $G$ be a flag-transitive point-primitive automorphism group of simple diagonal type with socle $T^m$. Then $k$ divides $\lambda m_1 h$, where $m_1\leq m$ and $h$ divides $|T|$.
\end{lemma}
\begin{proof}
See the proof of \rm{Proposition~3.1} in \cite{a:Zhou-lam100}.
\end{proof}

\begin{proposition}\label{prop:diag}
If $\Dmc=(\Pmc,\Bmc)$ is a non-trivial symmetric $(v, k, \lambda)$ design admitting a flag-transitive and point-primitive automorphism group $G$, where $\lambda$ divides $k$ and $k\geq \lambda^2$, then $G$ is not of simple diagonal type.
\end{proposition}
\begin{proof}
Suppose by contradiction that $G$ is a primitive group of simple diagonal type. Then $v=|T|^{m-1}$, and so by Lemma \ref{lem:six}(b), $\lambda v < k^2$. This implies that $\lambda |T|^{m-1} < k^2=\lambda^2t^2 $. Since $\lambda\leq k^2$, we must have $\lambda\leq t$, and hence
\begin{align}\label{eq:diag}
|T|^{m-1}< t^3.
\end{align}
Note by Lemma~\ref{lem:sub-sd} that $k$ divides $\lambda m_1h$ and $m_1h\leq m|T|$. Then $t$ divides $m_1h$, and so $t\leq m|T|$. We now apply  \eqref{eq:diag} and conclude that
$|T|^{m-1} < m^3|T|^3$. Therefore, $|T|^{m-4}<m^3$. Since $|T|\geq 60$, we must have $m< 6$. If $m=5$, then $|T|<5^{3}$, and this follows that $T\cong A_{5}$. Note that $k$ divides $\lambda(v-1)=\lambda(|T|^{m-1}-1)$. Then $t$ divides $|T|^{m-1}-1=60^{4}-1=13\cdot 59\cdot 61\cdot 277$. Since $t\leq m|T|=300$ and $t\geq 2$, it follows that $t\in\{13, 59, 61, 277\}$. For each such $t$, we have that $\lambda\leq t$ and $k=t\lambda$, and so we easily observe that these parameters does not satisfy Lemma~\ref{lem:six}(a). Therefore $m\in \{2,3,4\}$. Note that $G_{\alpha}$ is isomorphic to a subgroup of $\Aut(T ) \times S_m$. Then by Lemma~\ref{lem:six}(b), the parameter $k$ divides  $|G_{\alpha}|$, and so $k$ divides $(m!)\cdot |T |\cdot |\Out(T )|$. On the other hand, Lemma~\ref{lem:six}(a) implies that $k$ divides $\lambda (|T|^{m-1}-1)$, and so $t$ divides $|T|^{m-1}-1$ implying that $\gcd(t, |T|)=1$. Since $k$ divides $(m!)\cdot |T |\cdot |\Out(T )|$ and $t$ is a divisor of $k$, we conclude that $t$ divides $(m!)\cdot |\Out(T)| $. Recall by \eqref{eq:diag} that $|T |^{m-1}< t^3$. Therefore,
\begin{equation}\label{eq:diag-2}
|T |^{m-1}<(m!)^{3}\cdot |\Out(T)|^3,
\end{equation}
where $m\in \{2,3,4\}$.

If $m = 2$, then $|T |< 8\cdot |\Out(T )|^3$. If $m = 3$, then $|T |^2< 6^3|\Out(T )|^3$, and so
$|T |< 6^{\frac{3}{2}}|\Out(T )|$. If $m = 4$, then $|T |^3 < 24^3|\Out(T )|^3$, and $|T | < 24|\Out(T )|$. Thus
for $m\leq 4$, we always have
\begin{equation*}
|T | < 8\cdot |\Out(T )|^3,
\end{equation*}
where $T$ is a non-abelian finite simple group. We now apply  Lemma~\ref{lem:out} and conclude that $T$ is isomorphic to $A_5$ or $A_6$. If $m=2$, then since $t$ divides $|T|^{m-1}-1=|T|-1$, we have that $t$ divides $59$ or $359$ when $T$ is isomorphic to $A_5$ or $A_6$, respectively. Thus $(v, k, \lambda)= (60, 59\lambda, \lambda)$ or $(v, k, \lambda)=(360, 359\lambda, \lambda)$. Since $\lambda>1$, in each case , we conclude that $k>v$, which is a contradiction. For $m=3,4$,  since  $|\Out(A_5)|=2$ and $|\Out(A_6)|=4$, it follows from \eqref{eq:diag-2} that $|T|< 48$ or $|T|<96 $ when $T$ is isomorphic to $A_5$ or $A_6$, respectively, which is a contradiction.
\end{proof}

\section{Proof of the main result}

In this section, we prove Theorem~\ref{thm:main}. Suppose that $\Dmc=(\Pmc,\Bmc)$ is a non-trivial symmetric $(v, k, \lambda)$ design with $\lambda$ divides $k$ and $k\geq \lambda^2$. Suppose also that $G$ is a flag-transitive automorphism group of $\Dmc$.

\begin{proof}[Proof of Theorem \ref{thm:main}]
If $G$ is point-primitive, then by O'Nan-Scott Theorem \cite{a:LPS-Onan-Scott} and Propositions~\ref{prop:prod}, \ref{prop:tw} and \ref{prop:diag}, we conclude that $G$ is of  affine or almost simple type. Suppose now that $G$ is point-imprimitive. Then $G$ leaves invariant a non-trivial partition $\mathcal{C}$ of $\Pmc$ with $d$ classes of size $c$. By \cite[Theorem~1.1]{a:Praeger-imprimitive}, there is a constant $l$ such that, for each $B\in \Bmc$ and $\Delta \in \mathcal{C}$,  $|B \cap\Delta|\in \{0,l\}$ and one of the following holds:
\begin{enumerate}[{\rm (a)}]
    \item $k \leq \lambda(\lambda -3)/2$;
    \item $(v, k, \lambda) = (\lambda^2(\lambda +2), \lambda(\lambda + 1), \lambda)$ with $(c, d, l) = (\lambda^2,\lambda +2,\lambda)$ or $(\lambda +2,\lambda^2, 2)$;
    \item $(v, k,\lambda, c, d, l) = ( \dfrac{(\lambda+2)(\lambda^2-2\lambda+2)}{4} , \dfrac{\lambda^2}{2} , \lambda, \dfrac{\lambda+2}{2} , \dfrac{\lambda^2-2\lambda+2}{2} , 2)$, and either $\lambda \equiv 0 \mod 4$, or $\lambda = 2u^2$, where $u$ is odd, $u\geq 3$, and $2(u^2-1)$ is a square;
    \item $(v, k, \lambda, c, d, l) = (\dfrac{(\lambda + 6)(\lambda^2+4\lambda-1)}{4} , \dfrac{\lambda (\lambda+5)}{2} , \lambda, \lambda+6, \dfrac{\lambda^2+4\lambda-1}{4},3)$, where $\lambda\equiv 1$ or $3$ $\mod 6$.
\end{enumerate}
We easily observe that the cases (a) and (c) can be ruled out as $k\geq \lambda^{2}$. If case (d) occurs, then $\lambda(\lambda+5)/2=k\geq \lambda^{2}$ implying that $\lambda\leq 5$. Since $\lambda \equiv 1 \text{ or } 3 \mod{6}$, it follows that $\lambda=3$ for which $(v,k,\lambda,c,d,l)=(45,12,3,9,5,3)$ which satisfies the condition in Theorem~\ref{thm:main}(b). Therefore, the case (b) can occur as claimed.
\end{proof}


\bibliographystyle{amcjoucc}

\end{document}